\documentclass{compositio}

\usepackage{graphicx}
\usepackage[font=small,labelfont=bf]{caption}
\usepackage{epstopdf}
\usepackage{enumitem}
\usepackage{lipsum}
\usepackage{amsthm}
\usepackage{amsmath,amsfonts,amsthm,bm}
\usepackage{verbatim}
\usepackage{tikz}
\usetikzlibrary{positioning}
\usetikzlibrary{arrows}
\usetikzlibrary{decorations.markings}
\usepackage{yhmath}
\usepackage{relsize}
\usepackage{calc}
\usepackage{mathrsfs}

\newtheorem{theorem}{Theorem}[section]
\makeatletter
\newcommand{\proofcase}[2]{%
  \par
  \addvspace{\medskipamount}%
  \noindent\emph{Case #1: #2}%
  \@afterheading
}
\makeatother
\makeatletter
\newcommand{\proofstep}[2]{%
  \par
  \addvspace{\medskipamount}%
  \noindent\emph{Step #1: #2}%
  \@afterheading
}
\makeatother
\newtheorem{lemma}[theorem]{Lemma}
\newtheorem{corollary}[theorem]{Corollary}

\theoremstyle{definition}
\newtheorem{definition}[theorem]{Definition}
\newtheorem{conjecture}[theorem]{Conjecture}

\theoremstyle{remark}
\newtheorem{remark}[theorem]{Remark}

\numberwithin{equation}{section}



\makeatletter
\def\blfootnote{\xdef\@thefnmark{}\@footnotetext}
\makeatother

\begin{document}

\title{Virtually Unipotent Curves in Some Non-NPC Graph Manifolds}

\author{Sami Douba}
\email{sami.douba@mail.mcgill.ca}
\address{Department of Mathematics and Statistics, McGill University, Montreal, QC H3A 0B9}

\begin{abstract}
Let $M$ be a graph manifold containing a single JSJ torus $T$ and whose JSJ blocks are of the form $\Sigma \times S^1$, where $\Sigma$ is a compact orientable surface with boundary. We show that if $M$ does not admit a Riemannian metric of everywhere nonpositive sectional curvature, then there is an essential curve on $T$ such that any finite-dimensional linear representation of $\pi_1(M)$ maps an element representing that curve to a matrix all of whose eigenvalues are roots of $1$. In particular, this shows that $\pi_1(M)$ does not admit a faithful finite-dimensional unitary representation, and gives a new proof that $\pi_1(M)$ is not linear over any field of positive characteristic.
\end{abstract}

\maketitle

\section{Introduction}
A matrix $P \in \mathrm{GL}(n, \mathbb{F})$, where $\mathbb{F}$ is a field, is {\it unipotent} if $1$ is the only eigenvalue of $P$ over the algebraic closure $\overline{\mathbb{F}}$ of $\mathbb{F}$. We say a matrix $P \in \mathrm{GL}(n, \mathbb{F})$ is {\it virtually unipotent} if some power of $P$ is unipotent, that is, if the eigenvalues of $P$ are all roots of $1$ in $\overline{\mathbb{F}}$. The only matrix in $\mathrm{GL}(n, \overline{\mathbb{F}})$ that is both unipotent and diagonalizable is the identity matrix; thus, a matrix in $\mathrm{GL}(n, \overline{\mathbb{F}})$ that is both virtually unipotent and diagonalizable has finite order.

We begin with an observation about a group consisting entirely of unipotent matrices: the integral Heisenberg group $H$, defined as the subgroup of $\mathrm{GL}(3, \mathbb{R})$ consisting of the upper unitriangular integer matrices. One might ask if $H$ can be realized as a subgroup of $\mathrm{GL}(n, \mathbb{F})$ for some (possibly different) $n$ and $\mathbb{F}$ that consists entirely of diagonalizable matrices. The following remark contains an elementary argument that this is impossible. See the work of Button for a proof of a more general result \cite[Theorem~3.2]{button2019aspects}, and for a broader discussion on unipotent matrices in matrix groups.

\begin{remark}\label{heisenberg} Let
\[
x = \begin{pmatrix} 1 & 1 & 0 \\ 0 & 1 & 0 \\ 0 & 0 & 1 \end{pmatrix}, \quad y = \begin{pmatrix} 1 & 0 & 0 \\ 0 & 1 & 1 \\ 0 & 0 & 1 \end{pmatrix}, \quad z = \begin{pmatrix} 1 & 0 & 1 \\ 0 & 1 & 0 \\ 0 & 0 & 1 \end{pmatrix}
\]
and let $\rho: H \rightarrow \mathrm{GL}(n, \mathbb{F})$ be any representation. Up to replacing $\mathbb{F}$ with its algebraic closure and postconjugating $\rho$, we may assume that $\rho(z)$ has a block-diagonal structure
\[
\rho(z) = \mathrm{diag}(Z_1, \ldots, Z_k)
\]
where $Z_r \in \mathrm{GL}(n_r, \mathbb{F})$ is upper triangular with a unique eigenvalue $\lambda_r \in \mathbb{F}^*$, and that $\lambda_1, \ldots, \lambda_k$ are distinct. Since $\rho(x), \rho(y)$ commute with $\rho(z)$, each of the former preserves the generalized eigenspaces of $\rho(z)$ and thus has a block-diagonal structure
\begin{alignat*}{4}
\rho(x) &= \mathrm{diag}(X_1, \ldots, X_k) \\
\rho(y) &= \mathrm{diag}(Y_1, \ldots, Y_k)
\end{alignat*}
where $X_r, Y_r \in \mathrm{GL}(n_r, \mathbb{F})$. Since $z = [x,y]$, we have $Z_r = [X_r, Y_r]$, and so
\[
\lambda_r^{n_r} = \det Z_r = 1
\]
for $r = 1, \ldots, k$. We conclude that $\rho(z)$ is a virtually unipotent matrix. Since $z$ has infinite order (as does any nontrivial unipotent matrix with entries in a field of characteristic zero), it follows that $\rho(z)$ cannot be diagonalizable if $\rho$ is to be faithful.
\end{remark}

Remark \ref{heisenberg} motivates the following definition.

\begin{definition}\label{virtuallyunipotent}
An element $\gamma$ of an arbitrary group $\Gamma$ is {\it virtually unipotent} if any finite-dimensional linear representation of $\Gamma$ maps $\gamma$ to a virtually unipotent matrix. 
\end{definition}

\begin{remark}
If $\Gamma$ is a residually finite group, as are many groups of interest and, in particular, as is the fundamental group of any closed 3-manifold \cite{hempel1987residual}, then for any nontrivial element $\gamma \in \Gamma$, there is a finite-dimensional unitary representation $\rho$ of $\Gamma$ such that $\rho(\gamma)$ is nontrivial and hence, by diagonalizability of unitary matrices, not unipotent. Thus, for our purposes, it is not sensible to omit the word ``virtually" in Definition \ref{virtuallyunipotent}.
\end{remark}

\begin{remark}\label{cyclic}
If $\gamma$ is a virtually unipotent element of a group $\Gamma$, then any element in the conjugacy class of $\gamma$ is virtually unipotent in $\Gamma$. Moreover, if $\Gamma_0$ is an abelian subgroup of $\Gamma$ generated by virtually unipotent elements of $\Gamma$, then any element of $\Gamma_0$ is virtually unipotent in $\Gamma$. The latter follows from the fact that an abelian subgroup of $\mathrm{GL}(n, \mathbb{F})$, where $\mathbb{F}$ is an algebraically closed field, is conjugate to an upper triangular subgroup of $\mathrm{GL}(n, \mathbb{F})$ \cite[Theorem~1.1.5]{radjavi2012simultaneous}.
\end{remark}

\begin{remark}\label{commensurable}
Suppose $\Gamma_0$ is a finite-index normal subgroup of a group $\Gamma$, and that $\gamma$ is a virtually unipotent element of $\Gamma$. Then a generator $\gamma_0$ of $\langle \gamma \rangle \cap \Gamma_0$ is a virtually unipotent element of $\Gamma_0$. Indeed, let $\rho_0$ be a finite-dimensional linear representation of $\Gamma_0$. Then $\rho_0$ is a direct summand of the restriction $\rho \bigr|_{\Gamma_0}$, where $\rho$ is the representation induced by $\rho_0$ on $\Gamma$. Since $\rho(\gamma_0)$ is a virtually unipotent matrix, it follows that the same is true for $\rho_0(\gamma_0)$.
\end{remark}

\begin{remark}\label{distorted}
Lubotzky–Mozes–Raghunathan \cite[Prop.~2.4]{lubotzky2000word} showed that an element generating a distorted cyclic subgroup of a finitely generated group is virtually unipotent.
\end{remark}

Note that a finite-order element of any group is virtually unipotent. From Remark \ref{heisenberg} (or Remark \ref{distorted}), one observes that the integer Heisenberg group $H$, viewed as an abstract group, contains an {\it infinite-order} virtually unipotent element (namely, a generator of the center of $H$), and hence by Remark \ref{commensurable} so does the fundamental group of any closed 3-manifold with Nil geometry. In fact, the argument in Remark \ref{heisenberg} shows that if an element $\gamma$ of a group $\Gamma$ is a product of commutators of elements $\gamma_i \in \Gamma$ such that $\gamma$ commutes with the $\gamma_i$, then $\gamma$ is a virtually unipotent element of $\Gamma$. Thus, for example, an element of $\pi_1(M)$ representing a Seifert fiber of a closed 3-manifold $M$ with $\widetilde{\mathrm{SL}(2, \mathbb{R})}$ geometry is virtually unipotent in $\pi_1(M)$. 

A manifold is said to be {\it nonpositively curved (NPC)} if it admits a Riemannian metric of everywhere nonpositive sectional curvature. Closed 3-manifolds locally modeled on Nil or $\widetilde{\mathrm{SL}(2, \mathbb{R})}$ are not NPC \cite{gromoll1971some, yau1971fundamental, eberlein1982canonical}.
The purpose of this article is to exhibit nontrivial virtually unipotent elements within fundamental groups of non-NPC 3-manifolds of a different nature. 

\begin{theorem}\label{main}
Let $M$ be a connected closed orientable irreducible 3-manifold containing exactly one JSJ torus, and each of whose JSJ blocks is a product of $S^1$ with a 
surface. If $M$ is not NPC, then $\pi_1(M)$ contains a nontrivial virtually unipotent element. 
\end{theorem} 

We use a necessary and sufficient condition (Theorem \ref{npc}) for such 3-manifolds $M$ to be NPC due to Buyalo–Kobel'skii \cite{buyalo1995geometrization}, and independently Kapovich–Leeb \cite{kapovich1996actions} in the case that $M$ has two JSJ blocks. Our argument is similar to Button's proof that Gersten's free-by-cyclic group contains a nontrivial virtually unipotent element \cite[Theorem~4.5]{button2017free}. We remark that if $M$ is a 3-manifold as in the statement of Theorem \ref{main} that is not the mapping torus of an Anosov homeomorphism of the 2-torus, then it follows from \cite{kapovich19983} that all cyclic subgroups of $\pi_1(M)$ are undistorted.

An example of a 3-manifold $M$ as in the statement of Theorem \ref{main} is the mapping torus of a Dehn twist about an essential simple closed curve on a closed orientable surface of genus at least 2 \cite[Theorem~3.7]{kapovich1996actions}. In this case, our proof in fact shows that an element of $\pi_1(M)$ representing that curve is virtually unipotent.

Since a 3-manifold $M$ as in the statement of Theorem \ref{main} is aspherical, any nontrivial element of $\pi_1(M)$ has infinite order. We are interested in infinite-order virtually unipotent elements, since the existence of such an element within a group has interesting representation theoretic consequences for the group. The following definition is due to Button \cite[Definition~2.2]{button2017free}.

\begin{definition} 
A group $\Gamma$ is {\it NIU-linear} if there is a faithful finite-dimensional linear representation $\rho$ of $\Gamma$ such that $\rho(\Gamma)$ does not contain unipotent matrices of infinite order.
\end{definition}

Since unitary matrices are diagonalizable, any group admitting a faithful finite-dimensional unitary representation is NIU-linear (in fact, the image of such a representation will contain no nontrivial unipotent matrices). Furthermore, if $\mathbb{F}$ is a field of positive characteristic, then every unipotent element of $\mathbb{F}$ has finite order \cite[Proposition~2.1]{button2017free}, and so any group admitting a faithful finite-dimensional linear representation over such $\mathbb{F}$ is NIU-linear. Since a group containing an infinite-order virtually unipotent element is evidently not NIU-linear, such a group neither admits a faithful finite-dimensional unitary representation, nor a faithful finite-dimensional representation over a field of positive characteristic.

As a consequence of Theorem \ref{main} and the work of several others, we obtain the following corollary.

\begin{corollary}\label{cor}
Let $M$ be as in the statement of Theorem \ref{main} and let $\Gamma = \pi_1(M)$. Then the following are equivalent:
\begin{enumerate}[label=(\roman*)]
\item $M$ is NPC; \label{misnpc}
\item $\Gamma$ virtually embeds in a finitely generated right-angled Artin group; \label{raag}
\item $\Gamma$ admits a faithful finite-dimensional unitary representation; \label{unitary}
\item $\Gamma$ admits a faithful finite-dimensional linear representation over a field of positive characteristic; \label{positivechar}
\item $\Gamma$ is NIU-linear; \label{niu}
\item $\Gamma$ does not contain a nontrivial virtually unipotent element. \label{vue}
\end{enumerate}
\end{corollary}

We explain how Corollary \ref{cor} can be established using Theorem \ref{main}. For us, a {\it graph manifold} is a connected closed orientable irreducible non-Seifert 3-manifold all of whose JSJ blocks are Seifert. The manifolds described in the statement of Theorem \ref{main} can be thought of as the simplest examples of graph manifolds. 
That \ref{misnpc} implies \ref{raag} in Corollary \ref{cor} is due to Liu \cite{liu2013virtual}, who showed that the fundamental group of any NPC graph manifold is virtually a subgroup of a finitely generated right-angled Artin group (RAAG). Agol \cite{315430} showed that such a RAAG embeds in $\mathrm{U}(n)$ for some $n$, so that \ref{raag} implies \ref{unitary}. Moreover, the work of Berlai–de la Nuez Gonz\'alez \cite{berlai2019linearity} implies that a finitely generated RAAG admits a faithful finite-dimensional linear representation over a field of positive characteristic (indeed, {\it any} prime characteristic), and so \ref{raag} also implies \ref{positivechar}. It was discussed before the statement of Corollary \ref{cor} that each of \ref{unitary} and \ref{positivechar} imply \ref{niu}, and that \ref{niu} implies \ref{vue}. Finally, the fact that \ref{vue} implies \ref{misnpc} is precisely the statement of Theorem \ref{main}.

\begin{remark}
Note that if $\gamma$ is an infinite-order virtually unipotent element of a group $\Gamma$, then $\langle \gamma \rangle \subset \Gamma$ is not a virtual retract of $\Gamma$ by Remark \ref{commensurable}. Thus, the fact that statement \ref{raag} in Corollary \ref{cor} implies statement \ref{vue} also follows from work of Minasyan \cite{10.1093/imrn/rnz249}.
\end{remark}

\begin{remark}
It was already known that statement \ref{positivechar} in Corollary \ref{cor} does not hold for the fundamental group $\Gamma$ of any non-NPC graph manifold. Indeed, Button \cite{button2019aspects} proved that any finitely generated group $\Gamma$ satisfying \ref{positivechar} acts properly by semisimple isometries on a complete CAT(0) metric space, and Leeb \cite{leeb1992manifolds} showed that if the fundamental group of a graph manifold $M$ admits such an action, then $M$ is NPC. At the time of writing of this article, it is not known if a single non-NPC graph manifold without Sol geometry admits a faithful finite-dimensional linear representation over a field of characteristic zero.
\end{remark}

\begin{remark}
It follows from the Lie–Kolchin–Mal'cev theorem that if $\Gamma$ is a polycyclic group then $\Gamma$ is NIU-linear if and only if $\Gamma$ is virtually abelian \cite{button2017properties}. Thus, for example, closed 3-manifolds with Nil or Sol geometry do not have NIU-linear fundamental groups. However, this handy obstruction to NIU-linearity is not useful for dealing with most 3-manifolds $M$ as in the statement of Theorem \ref{main} since, for all such $M$ apart from the mapping torus of an Anosov homeomorphism of the 2-torus, any subgroup of $\pi_1(M)$ lacking a nonabelian free subgroup is in fact abelian \cite{frigerio2011rigidity}.
\end{remark}

Statements \ref{misnpc}-\ref{niu} in Corollary \ref{cor} were known to be equivalent for $M$ a closed aspherical geometric 3-manifold and $\Gamma = \pi_1(M)$ \cite{button2017free}. It is also true that \ref{vue} is equivalent to \ref{misnpc}-\ref{niu} in this case. Indeed, to argue this, it suffices to show that if $M$ is a closed aspherical geometric non-NPC 3-manifold, that is, if $M$ is a closed 3-manifold with Nil, $\widetilde{\mathrm{SL}(2, \mathbb{R})}$, or Sol geometry, then $\pi_1(M)$ contains a nontrivial virtually unipotent element. This was already justified for the first two geometries, and closed Sol 3-manifolds are handled by Theorem \ref{main}. 

We conjecture the following.

\begin{conjecture}\label{conj}
Statements \ref{misnpc}-\ref{vue} in Corollary \ref{cor} are equivalent for any closed aspherical 3-manifold $M$ and $\Gamma = \pi_1(M)$.
\end{conjecture}

By work of Agol \cite{agol2013virtual} and Przytycki–Wise \cite{przytycki2018mixed}, and the aforementioned work of Liu \cite{liu2013virtual}, the fundamental group $\Gamma$ of any closed NPC 3-manifold satisfies \ref{raag} and hence \ref{unitary}-\ref{vue} in Corollary \ref{cor}. Furthermore, a closed aspherical non-geometric non-NPC 3-manifold is (up to passing to its orientation cover) a graph manifold \cite{leeb1992manifolds}. Thus, Conjecture \ref{conj} amounts to the claim that the fundamental group of any non-NPC graph manifold contains a nontrivial virtually unipotent element.

\subsection*{Organization} In Section \ref{preliminaries}, we fix the language in which we prove Theorem \ref{main}, and also present two lemmas that will be useful in the proof. The proof of Theorem \ref{main} is contained in Section \ref{proofofmain} and is divided into two cases: the case that there is a single block in the JSJ decomposition of the 3-manifold $M$ (Theorem \ref{looptheorem}), and the case that there are two (Theorem \ref{edgetheorem}). The two proofs are very similar and are somewhat technical, but involve only elementary linear algebra.

\subsection*{Acknowledgements} I am grateful to my supervisor Piotr Przytycki for his encouragement and guidance, and in particular for suggesting elegant techniques for taming large systems of equations. I also thank Jack Button for helpful discussions.

\section{Preliminaries}\label{preliminaries}

\subsection{Definitions}

If $S$ is a (not necessarily connected) closed surface embedded in a 3-manifold $M$, we denote by~$M \bigr| S$ the complement in $M$ of a small open tubular neighborhood of $S$. If $M$ is a connected closed orientable irreducible 3-manifold, then there is, up to isotopy, a unique minimal collection $\mathcal{E}$ of disjoint embedded incompressible tori such that each component of $M \bigr| \bigcup \mathcal{E}$ is either Seifert or atoroidal (see, for example, \cite[Thm~1.41]{kapovich2009hyperbolic} and the references therein). The decomposition of $M$ into the components of $M \bigr| \bigcup \mathcal{E}$ is called the {\it Jaco–Shalen–Johannson (JSJ) decomposition} of $M$. If $\mathcal{E} = \emptyset$, we say $M$ has {\it trivial} JSJ decomposition. Note that if $M$ is the mapping torus of an Anosov homeomorphism of the 2-torus, then $M$ has nontrivial JSJ decomposition.

Let $\mathfrak{G}$ denote the class of all connected closed orientable irreducible non-Seifert 3-manifolds~$M$ such that each component $M_v$ of $M \bigr| \bigcup \mathcal{E}$, where $\mathcal{E}$ is the collection of JSJ tori in $M$, is a trivial $S^1$-bundle over a compact orientable surface $\Sigma_v$ with boundary. The manifolds $M_v$ are the {\it blocks} of $M$. 
The {\it underlying graph} $\mathcal{G} = \mathcal{G}(M)$ of $M$ is the graph dual to the JSJ decomposition of $M$; the graph $\mathcal{G}$ is well-defined since the collection $\mathcal{E}$ is unique up to isotopy. We identify the vertex set $\mathcal{V}$ of $\mathcal{G}$ with the set of blocks of $M$, and the set of unoriented edges of $\mathcal{G}$ with $\mathcal{E}$. Denote by~$\mathcal{W}$ the set of oriented edges of $\mathcal{G}$. We identify $\mathcal{W}$ with the set of boundary components of $M \bigr| \bigcup \mathcal{E}$ by assigning to each oriented edge $w \in \partial v \subset \mathcal{W}$ the corresponding boundary component $T_w$ of $M_v$. 

Choose an orientation of $M$, thereby inducing an orientation on each block $M_v$ of $M$, and hence on each component of $\partial M_v$. For each $v \in \mathcal{V}$, choose an orientation of the fibers in $M_v$, as well as a {\it Waldhausen basis} for $H_1(\partial M_v ; \mathbb{Z})$; that is, a basis $\{(f_w, z_w) \> | \> w \in \partial v\}$ for $H_1(\partial M_v ; \mathbb{Z}) = \bigoplus_{w \in \partial v} H_1(T_w ; \mathbb{Z})$ such that the elements $f_w$ represent oriented fibers,  the algebraic intersection number $\hat{i}(z_w, f_w)$ on $T_w$ is $+1$, and the sum $\oplus_{w \in \partial v}z_w$ lies in the kernel of the map $H_1(\partial M_v ; \mathbb{Z}) \rightarrow~H_1(M_v ; \mathbb{Z})$ induced by inclusion. We call the additional structure on $M$ given by the choices made in this paragraph a {\it framing} of $M$.

An oriented edge $w \in \mathcal{W}$ corresponds to a gluing homeomorphism $T_{-w} \rightarrow T_w$, which induces an isomorphism $\phi_w: H_1(T_{-w}; \mathbb{Z}) \rightarrow H_1(T_w; \mathbb{Z})$. Define $B_w = \begin{pmatrix} a_w & b_w \\ c_w & d_w \end{pmatrix} \in \mathrm{GL}(2,\mathbb{Z})$ to be the matrix whose entries satisfy
\begin{alignat*}{4}
\phi_w(f_{-w}) &= a_w f_w + b_w z_w \\
\phi_w(z_{-w}) &= c_w f_w + d_w z_w
\end{alignat*}
Note that $\det B_w = -1$ since $M$ is orientable, that $B_{-w} = B_w^{-1}$, and that $b_w \neq 0$ by minimality of $\mathcal{E}$.

This article is concerned with the subclasses $\mathfrak{E}, \mathfrak{L}$ of $\mathfrak{G}$ consisting of all manifolds $M$ in $\mathfrak{G}$ whose underlying graph is a single edge (joining distinct vertices) or a loop, respectively. We call~$B \in \mathrm{GL}(2, \mathbb{Z})$ a {\it gluing matrix} for such a manifold $M$ if $B = B_w$ for an oriented edge $w$ of~$\mathcal{G}(M)$ with respect to some framing of $M$. 

The fundamental group $\pi_1(M)$ of a manifold $M \in \mathfrak{E}$ with gluing matrix $B = \begin{pmatrix} a & b \\ c & d \end{pmatrix}$ and whose blocks $M_v, M_{v'}$ have base surfaces $\Sigma, \Sigma'$ of genus $g, g'$, respectively, is isomorphic to the group $\Gamma_{g,g',B}^\mathfrak{E}$ given by the presentation with generators 
\begin{alignat*}{4}
 & x_1, y_1, \ldots, x_g, y_g, z, f, \\ & x'_1, y'_1, \ldots, x'_{g'}, y'_{g'}, z', f'
 \end{alignat*}
subject to the relations
\begin{enumerate}[label=(\Roman*)]
\item $z = \prod_{i=1}^g [x_i, y_i]$, \label{block1productofcommutators}
\item $[x_i, f] = [y_i, f] = 1$ for $i = 1, \ldots, g$, \label{block1commuting}
\item $z' = \prod_{i=1}^{g'}[x'_i, y'_i]$, 
\item $[x'_i, f'] = [y'_i, f'] = 1$ for $i = 1, \ldots, g'$, \label{block2commuting}
\item $f' = f^az^b$, \label{firstrowedge}
\item $z' = f^cz^d$, \label{secondrowedge}
\end{enumerate}
where the subgroup $\langle x_1, y_1, \ldots, x_g, y_g \rangle$ (resp., $\langle x'_1, y'_1, \ldots, x'_g, y'_g \rangle$) is the image of the map $\pi_1(\Sigma) \rightarrow\pi_1(M)$ (resp., $\pi_1(\Sigma') \rightarrow \pi_1(M)$) induced by the inclusions $\Sigma \subset M_v \subset M$ (resp., $\Sigma' \subset M_{v'} \subset M$), and the element $f$ (resp., $f'$) represents an oriented fiber of $M_v$ (resp., $M_{v'}$). 
\begin{remark}\label{manipulatematrix}
Note that if $C$ is obtained from $B$ by negating a row or a column of $B$, then $\Gamma_{g,g',B}^\mathfrak{E} \cong \Gamma_{g,g',C}^\mathfrak{E}$. Note also that $\Gamma_{g,g',B}^\mathfrak{E} \cong \Gamma_{g,g',B^{-1}}^\mathfrak{E}$.
\end{remark}
The fundamental group $\pi_1(M)$ of a manifold $M \in \mathfrak{L}$ with gluing matrix $B = \begin{pmatrix} a & b \\ c & d \end{pmatrix}$ and the base surface $\Sigma$ of whose unique block $M_v$ has genus $g$ is isomorphic to the group $\Gamma_{g, B}^\mathfrak{L}$ given by the presentation with generators \[ x_1, y_1, \ldots, x_g, y_g, z, z', f, t \] subject to the relations
\begin{enumerate}[label=(\arabic*)]
\item $zz' = \prod_{i=1}^g [x_i, y_i]$, \label{productofcommutatorsloop}
\item $[x_i, f] = [y_i, f] = [z, f ] = 1$ for $i = 1, \ldots, g$, \label{commutingloop}
\item $tft^{-1} = f^a z^b$, \label{firstrowloop}
\item $tz't^{-1} = f^c z^d$, \label{secondrowloop}
\end{enumerate}
where the subgroup $\langle x_1, y_1, \ldots, x_g, y_g, z \rangle$ is the image of the map $\pi_1(\Sigma) \rightarrow \pi_1(M)$ induced by the inclusion $\Sigma \subset M_v \subset M$, and the element $f$ represents an oriented fiber of $M_v$.

\begin{remark}\label{inverse}
Note that $\Gamma_{g, B}^\mathfrak{L} \cong \Gamma_{g, B^{-1}}^\mathfrak{L}$.
\end{remark}

The following theorem is a special case of a result of Buyalo–Kobel'skii \cite{buyalo1995geometrization}, and was proved independently by Kapovich–Leeb  \cite{kapovich1996actions} in the case $M \in \mathfrak{E}$. 

\begin{theorem}\label{npc}
Let $M \in \mathfrak{E}$ (resp., $M \in \mathfrak{L}$) and let $B = \begin{pmatrix} a & b \\ c & d \end{pmatrix} \in \mathrm{GL}(2, \mathbb{Z})$ be a gluing matrix for $M$. Then $M$ is NPC if and only if $a=d=0$ (resp., if and only if $|a-d| \geq 2)$.
\end{theorem}

\subsection{Basic lemmas}

The following lemma will allow us to conjugate a representation $\rho$ of the appropriate group $\Gamma$ in a manner that makes the interactions between generalized eigenspaces of certain elements of~$\rho(\Gamma)$ more apparent.

\begin{lemma}\label{triangular}
Let $\mathbb{F}$ be an algebraically closed field, let $P, P', Q \in \mathrm{M}_{n\times n}(\mathbb{F})$, and let $\lambda_1, \ldots, \lambda_k$ (resp. $\lambda'_1, \ldots, \lambda'_\ell$) be the distinct eigenvalues of $P$ (resp. $P'$).  If $P,P',Q$ pairwise commute, then there is a single matrix $C \in \mathrm{GL}(n,\mathbb{F})$ such that
\begin{alignat*}{4}
CPC^{-1} &= \mathrm{diag}\left(P_{1,1}, \ldots, P_{1,\ell}, \ldots, P_{k,1}, \ldots, P_{k,\ell} \right) \\
CP'C^{-1} &= \mathrm{diag}\left(P'_{1,1}, \ldots, P'_{1,\ell}, \ldots, P'_{k,1}, \ldots, P'_{k,\ell} \right) \\
CQC^{-1} &= \mathrm{diag}\left(Q_{1,1}, \ldots, Q_{1,\ell}, \ldots, Q_{k,1}, \ldots, Q_{k,\ell} \right)
\end{alignat*}
where $P_{r,s}, P'_{r,s}, Q_{r,s}$ are (possibly empty) upper triangular matrices and the only eigenvalue of~$P_{r,s}$ (resp., $P'_{r,s}$) is $\lambda_r$ (resp., $\lambda'_s$). 
\end{lemma}

\begin{proof} 
For $r=1, \ldots, k$, let $W_r$ be the generalized $\lambda_r$-eigenspace of $P$, and let $n_r = \dim W_r$. We index the standard ordered basis for $\mathbb{F}^n$ as follows:
\[
(e_{1,1}, \ldots, e_{1,n_1}, \ldots, e_{k,1}, \ldots, e_{k, n_k})
\]
We may assume that $W_r = \mathrm{Span}(e_{r,1}, \ldots, e_{r, n_r})$. Since each of $P', Q$ commutes with $P$, we have that $P',Q$ preserve the generalized eigenspaces of $P$, so $P, P', Q$ share a block-diagonal structure
\begin{alignat*}{4}
P &= \mathrm{diag}(P_1, \ldots, P_k) \\
P' &= \mathrm{diag}(P'_1, \ldots, P'_k) \\
Q &= \mathrm{diag}(Q_1, \ldots, Q_k)
\end{alignat*}
where $P_r, P'_r \in \mathrm{M}_{n_r\times n_r}(\mathbb{F})$. We may also assume that for some indexing
\[
(e_{r,1,1}, \ldots, e_{r,1,n_{r,1}}, \ldots, e_{r, \ell, 1}, \ldots, e_{r, \ell, n_{r,\ell}})
\]
of the ordered basis $(e_{r,1}, \ldots, e_{r,n_r})$ for  $W_r$, where the $n_{r,s}$ are nonnegative integers satisfying $\sum_{s=1}^\ell n_{r,s} = n_r$, we have that $\mathrm{Span}(e_{r,s,1}, \ldots, e_{r,s, n_{r,s}})$ is the generalized $\lambda'_s$-eigenspace of $P'_r$. Since each of $P_r, Q_r$ commutes with $P'_r$, we have that $P_r, Q_r$ preserve the generalized eigenspaces of $P'_r$, so $P_r, P'_r, Q_r$ share a block-diagonal structure
\begin{alignat*}{4}
P_r &= \mathrm{diag}(P_{r,1}, \ldots, P_{r,\ell}) \\
P'_r &= \mathrm{diag}(P'_{r,1}, \ldots, P'_{r,\ell}) \\
Q_r &= \mathrm{diag}(Q_{r,1}, \ldots, Q_{r,\ell})
\end{alignat*}
Now since $P_{r,s}, P'_{r,s}, Q_{r,s}$ pairwise commute, they are simultaneously upper triangularizable \cite[Theorem~1.1.5]{radjavi2012simultaneous}, and Lemma \ref{triangular} follows.
\end{proof}

The following lemma will allow us to reduce systems of equations whose unknowns lie in~$\mathbb{F}^*$, where $\mathbb{F}$ is some field, to systems of linear equations with integer unknowns. It is a step in the proof of Theorem 4.5 in \cite{button2017free}. We include Button's argument for the convenience of the reader.

\begin{lemma}\label{abelian}
Let $M$ be an integer matrix with $L$ columns and suppose there is a subset $I \subset \{1, \ldots, L\}$ such that for any $\bm{ \alpha} = (\alpha_1, \ldots, \alpha_L)^T \in \mathbb{Z}^L$ satisfying $M \bm{\alpha} = 0$, we have $\alpha_i = 0$ for~$i \in I$. Let $A$ be a torsion-free abelian group,  and suppose ${\bf a} = (a_1, \ldots, a_L)^T \in A^L$ satisfies $M {\bf a} = 0$. Then $a_i = 0$ for $i \in I$.
\end{lemma}

\begin{proof}
Let $A_0 = \langle a_1, \ldots, a_L \rangle \subset A$. Then $A_0$ is a finitely generated torsion-free abelian group, so there is an isomorphism $\varphi: A_0 \rightarrow \mathbb{Z}^K$ for some $K$. For each $j = 1, \ldots, K$, we have
\[
M (\varphi_j(a_1), \ldots, \varphi_j(a_L))^T = 0
\] where $\varphi_j = p_j \circ \varphi$ and $p_j: \mathbb{Z}^K \rightarrow \mathbb{Z}$ is the projection onto the $j$th coordinate, so that $\varphi_j(a_i) = 0$ for $i \in I$. We conclude that $\varphi(a_i) = 0$, and hence $a_i = 0$, for $i \in I$.
\end{proof}

\section{Proof of Theorem \ref{main}}\label{proofofmain}

We divide Theorem \ref{main} into Theorem \ref{looptheorem} (the loop case) and Theorem \ref{edgetheorem} (the edge case), and prove each separately. 

\begin{theorem}\label{looptheorem}
Suppose $M \in \mathfrak{L}$ is not NPC, and let $\Gamma = \pi_1(M)$. Then $\Gamma$ contains a nontrivial virtually unipotent element. 
\end{theorem}
\begin{proof}
By Remark \ref{inverse} and Theorem \ref{npc}, we have $\Gamma = \Gamma_{g,B}^\mathfrak{L}$ for some $g\geq 0$ and $B = \begin{pmatrix} a & b \\ c & d \end{pmatrix} \in \mathrm{GL}(2,\mathbb{Z})$ with $\det B = -1$, $b \neq 0$, and $a - d \geq 2$. We show that $f^{a-1}z^b \in \Gamma$ is virtually unipotent.

Let $\mathbb{F}$ be an algebraically closed field, $n \geq 1$, and $\rho: \Gamma \rightarrow \mathrm{GL}(n, \mathbb{F})$ any representation. We may assume that $\rho$ is indecomposable. Let $\lambda_1, \ldots, \lambda_k \in \mathbb{F}^*$ be the distinct eigenvalues of $\rho(f)$, and let $f' = tft^{-1}$. By Lemma \ref{triangular} and relation \ref{firstrowloop} in the presentation of $\Gamma$, we may assume further that
\begin{alignat*}{5}
\rho(f)  &= \mathrm{diag} ( F_{1,1}, \ldots, F_{1,k}, \ldots, F_{k,1}, \ldots, F_{k,k}) \\
\rho(f') &= \mathrm{diag} ( F'_{1,1}, \ldots, F'_{1,k}, \ldots, F'_{k,1}, \ldots, F'_{k,k}) \\
\rho(z) &=  \mathrm{diag} ( D_{1,1}Z_{1,1}, \ldots, D_{1,k}Z_{1,k}, \ldots, D_{k,1}Z_{k,1}, \ldots, D_{k,k}Z_{k,k})
\end{alignat*}
where $F_{r,s}, F'_{r,s}, Z_{r,s} \in \mathrm{GL}(n_{r,s}, \mathbb{F})$ are (possibly empty) upper triangular matrices, $D_{r,s}$ is a (possibly empty) diagonal matrix in $  \mathrm{GL}(n_{r,s}, \mathbb{F})$ whose diagonal entries are $|b|^{\text{th}}$ roots of $1$ in $\mathbb{F}$, and the only eigenvalue of $F_{r,s}$ (resp., $F'_{r,s}, Z_{r,s}$) is $\lambda_r$ (resp., $\lambda_s, \mu_{r,s}$), with $\mu_{r,s} \in \mathbb{F}^*$ satisfying
\begin{equation}\label{equation1}
\lambda_s = \lambda_r^a \mu_{r,s}^b.
\end{equation}
Since, by relation \ref{commutingloop} in the presentation of $\Gamma$, each of $\rho(z'), \rho(x_1), \rho(y_1), \ldots, \rho(x_g), \rho(y_g)$ commutes with $\rho(f)$, each preserves the generalized eigenspaces of $\rho(f)$, and so
\begin{alignat*}{5}
\rho(z') &= \mathrm{diag}(Z'_1, \ldots, Z'_k) \\
\rho(x_i) &= \mathrm{diag}(X_1^{(i)}, \ldots, X_k^{(i)}) \\
\rho(y_i) &= \mathrm{diag}(Y_1^{(i)}, \ldots, Y_k^{(i)}) 
\end{alignat*}
for some $Z'_r, X_r^{(i)}, Y_r^{(i)} \in \mathrm{GL}(n_r, \mathbb{F})$, where $n_r = \sum_{s=1}^k n_{r,s}$ is the dimension of the generalized $\lambda_r$-eigenspace of $\rho(f)$.

Let $V_r$ be the generalized $\lambda_r$-eigenspace of $\rho(f')$. Then $\rho(t)^{-1}V_r$ is the generalized $\lambda_r$-eigenspace of $\rho(t)^{-1}\rho(f')\rho(t) = \rho(f)$, and the characteristic polynomial of \[ \rho(z') \bigr|_{\rho(t)^{-1}V_r} = \rho(t)^{-1}\rho(f^ch^d)\rho(t) \bigr|_{\rho(t)^{-1}V_r} \] coincides with the characteristic polynomial of $\rho(f^cz^d)\bigr|_{V_r}$. Thus, up to multiplying each root by a root of $1$, the characteristic polynomial of the block $Z'_r$ is 
\[
(x - \lambda_1^c \mu_{1,r}^d)^{n_{1,r}}\ldots (x - \lambda_k^c\mu_{k,r}^d)^{n_{k,r}}.
\]

Now let $Z_r = \mathrm{diag}(D_{r,1}Z_{r,1}, \ldots, D_{r,k}Z_{r,k})$. Then, by relation \ref{productofcommutatorsloop} in the presentation of $\Gamma$, we have $Z_r Z'_r = \prod_{i=1}^g [X_r^{(i)}, Y_r^{(i)}]$, so that $\det(Z_rZ'_r) = 1$. It follows that
\begin{equation}\label{equation2}
\prod_{s=1}^k\mu_{r,s}^{n_{r,s}}(\lambda_s^c\mu_{s,r}^d)^{n_{s,r}} = 1
\end{equation}
in the quotient $A$ of the group of units $\mathbb{F}^*$ by its torsion subgroup. Viewing (\ref{equation1}) also as equations in $A$ and switching to additive notation within $A$, we obtain the equations
\begin{alignat}{5}
& \lambda_s = a \lambda_r + b \mu_{r,s} \label{equation1additiveloop} \\
& \sum_{s=1}^k  \left( n_{r,s}\mu_{r,s} +n_{s,r}(c\lambda_s + d\mu_{s,r})   \right) = 0. \label{equation2additiveloop}
\end{alignat}
Multiplying (\ref{equation2additiveloop}) by $b$ and substituting $\lambda_s - a\lambda_r$ for $b\mu_{r,s}$, we have
\[
\sum_{s=1}^k \biggr(  n_{r,s}(\lambda_s - a \lambda_r) + n_{s,r}\bigr(bc\lambda_s + d(\lambda_r - a\lambda_s)\bigr) \biggr) = 0
\]
and so
\begin{equation}\label{equation3}
\sum_{s=1}^k \bigr(n_{r,s} + (bc-ad)n_{s,r})\lambda_s = \lambda_r \sum_{s=1}^k (an_{r,s} - dn_{s,r}).
\end{equation}
Since $bc - ad = - \det B = 1$, the left-hand side of (\ref{equation3}) is equal to $\sum_{s=1}^k (n_{r,s} +n_{s,r} )\lambda_s$. On the other hand, since $\sum_{s=1}^k n_{s,r} = \sum_{s=1}^k n_{r,s} = n_r$, the right-hand side of (\ref{equation3}) is equal to $(a-d)n_r\lambda_r$. 

In summary, $\lambda_1, \ldots, \lambda_k$ satisfy
\begin{equation}\label{equation4}
\sum_{s=1}^k (n_{r,s} + n_{s,r})\lambda_s = (a-d)n_r\lambda_r
\end{equation}
as elements of $A$. We now show that if we set $A = \mathbb{Z}$, then (\ref{equation4}) implies $\lambda_1 = \ldots = \lambda_k$, so that
\[
(a-1)\lambda_r + b\mu_{r,s} = a \lambda_r - \lambda_s + b\mu_{r,s} = 0
\]
where the second equality follows from (\ref{equation1additiveloop}). By Lemma \ref{abelian}, it will follow that $(a-1)\lambda_r + b\mu_{r,s} = 0$ in the original torsion-free abelian group $A$, thus completing the proof.

To that end, suppose for a contradiction that the integers $\lambda_1, \ldots, \lambda_k$ are not all equal. Then we may assume \[\lambda_1 = \ldots = \lambda_{r_0} > \lambda_{r_0+1}, \ldots, \lambda_k\] for some $r_0 \in \{1, \ldots, k-1\}$. Thus, for $r=1, \ldots, r_0$, either we have $n_{r,s} + n_{s,r} = 0$ for $s > r_0$, or we obtain the contradiction
\begin{alignat*}{4}
2n_r \lambda_r = \sum_{s=1}^k (n_{r,s}+n_{s,r})\lambda_r > \sum_{s=1}^k (n_{r,s} + n_{s,r})\lambda_s = (a-d)n_r \lambda_r \geq 2n_r\lambda_r.
\end{alignat*}
We conclude that $n_{r,s} = n_{s,r} = 0$ for $r \leq r_0$ and $s > r_0$, so that $\rho(t)$ preserves the span of the first $\sum_{r=1}^{r_0}n_r$ standard basis vectors and the span of the last $\sum_{r=r_0+1}^kn_r$ standard basis vectors of $\mathbb{F}^n$. But then $\rho(\Gamma)$ also preserves each of these subspaces, contradicting the indecomposability of $\rho$.
\end{proof} \vspace{5mm}

\begin{theorem}\label{edgetheorem}
Suppose $M \in \mathfrak{E}$ is not NPC, and let $\Gamma = \pi_1(M)$. Then $\Gamma$ contains a nontrivial virtually unipotent element. 
\end{theorem}

\begin{proof}
We have $\Gamma = \Gamma_{g,g',B}^\mathfrak{E}$ for some $g, g' \geq 1$, where $B = \begin{pmatrix} a & b \\ c & d \end{pmatrix}$ is a gluing matrix for $M$. Note that $b \neq 0$, and that, by Theorem \ref{npc}, one of $a,d$ is nonzero. By Remark \ref{manipulatematrix}, up to replacing $B$ with its inverse, we may assume $a \neq 0$. Furthermore, by Remark \ref{manipulatematrix} and the fact that $|\det B | = 1$, up to negating rows and columns of $B$, we may assume $a,b,c,d \geq 0$. We show that if $c =0$ (resp., $c > 0$) then $z$ (resp., $f$) is a virtually unipotent element of $\Gamma$.

Let $\mathbb{F}$ be an algebraically closed field, $n \geq 1$, and $\rho: \Gamma \rightarrow \mathrm{GL}(n, \mathbb{F})$ any representation. Let $\lambda_1, \ldots, \lambda_k \in \mathbb{F}^*$ (resp., $\lambda'_1, \ldots, \lambda'_\ell \in \mathbb{F}^*$) be the distinct eigenvalues of $\rho(f)$ (resp., $\rho(f')$). By Lemma \ref{triangular} and relation \ref{firstrowedge} in the presentation of $\Gamma$, we may assume that
\begin{alignat*}{5}
\rho(f)  &= \mathrm{diag} ( F_{1,1}, \ldots, F_{1,\ell}, \ldots, F_{k,1}, \ldots, F_{k,\ell}) \\
\rho(f') &= \mathrm{diag} ( F'_{1,1}, \ldots, F'_{1,\ell}, \ldots, F'_{k,1}, \ldots, F'_{k,\ell}) \\
\rho(z) &=  \mathrm{diag} ( D_{1,1}Z_{1,1}, \ldots, D_{1, \ell}Z_{1,\ell}, \ldots, D_{k,1}Z_{k,1}, \ldots, D_{k, \ell}Z_{k,\ell})
\end{alignat*}
where $F_{r,s}, F'_{r,s}, Z_{r,s} \in \mathrm{GL}(n_{r,s}, \mathbb{F})$ are (possibly empty) upper triangular matrices, $D_{r,s}$ is a diagonal matrix in $\mathrm{GL}(n_{r,s}, \mathbb{F})$ whose diagonal entries are $b^\text{th}$ roots of $1$ in $\mathbb{F}$, and the only eigenvalue of $F_{r,s}$ (resp., $F'_{r,s}, Z_{r,s}$) is $\lambda_r$ (resp., $\lambda'_s, \mu_{r,s}$), with $\mu_{r,s} \in \mathbb{F}^*$ satisying
\begin{equation}\label{equation1edge}
\lambda'_s = \lambda_r^a \mu_{r,s}^b.
\end{equation}

Since, by relation \ref{block1commuting} in the presentation of $\Gamma$, the $\rho(x_i), \rho(y_i)$ commute with $\rho(f)$, each of the former preserves the eigenspaces of $\rho(f)$. Thus, we have
\begin{alignat*}{4}
\rho(x_i) &= \mathrm{diag}(X_1^{(i)}, \ldots, X_k^{(i)}) \\
\rho(y_i) &= \mathrm{diag}(Y_1^{(i)}, \ldots, Y_k^{(i)})
\end{alignat*}
for some $X_r^{(i)}, Y_r^{(i)} \in \mathrm{GL}(n_r, \mathbb{F})$, where $n_r = \sum_{s=1}^k n_{r,s}$ is the dimension of the generalized $\lambda_r$-eigenspace of $\rho(f)$. Letting $Z_r = \mathrm{diag}(D_{r,1}Z_{r,1}, \ldots, D_{r,\ell}Z_{r, \ell})$, we have by relation \ref{block1productofcommutators} in the presentation of $\Gamma$ that
\[
Z_r = \prod_{i=1}^g [X_r^{(i)}, Y_r^{(i)}]
\]
for $r = 1, \ldots, k$. Thus, $\det Z_r = 1$, and so
\begin{equation}\label{equation2edge}
\prod_{s=1}^\ell \mu_{r,s}^{n_{r,s}} = 1
\end{equation}
in the quotient $A$ of $\mathbb{F}^*$ by its torsion subgroup. 

Since, by relation \ref{block2commuting} in the presentation of $\Gamma$, the $\rho(x'_i), \rho(y'_i)$ commute with $\rho(f')$, each of the former preserves the eigenspaces of $\rho(f')$. Thus, by a similar argument to the one given above, and by relation \ref{secondrowedge} in the presentation of $\Gamma$, we have
\begin{equation}\label{equation3edge}
\prod_{r=1}^k (\lambda_r^c\mu_{r,s}^d)^{n_{r,s}} = 1
\end{equation}
in $A$ for $s =1, \ldots, \ell$. Switching to additive notation within $A$, we obtain from (\ref{equation1edge}), (\ref{equation2edge}), (\ref{equation3edge}) the equations
\begin{alignat}{5}
& a \lambda_1 + b \mu_{1,s} = \ldots = a \lambda_k + b \mu_{k,s} \> \text{ for }s=1, \ldots, \ell, \label{equation4additive} \\
& \sum_{s = 1}^\ell n_{r,s}\mu_{r,s} = 0 \> \text{ for }r=1, \ldots, k, \label{equation2additive}\\
& \sum_{r=1}^k n_{r,s}(c\lambda_r + d\mu_{r,s}) = 0 \> \text{ for }s=1, \ldots, \ell. \label{equation3additive}
\end{alignat}
We now set $A = \mathbb{Z}$ and show that, in this context, equations (\ref{equation2additive}), (\ref{equation3additive}), and (\ref{equation4additive}) imply that if $c = 0$ (resp., $c > 0$) then $\mu_{r,s} = 0$ whenever $n_{r,s} > 0$ (resp., then $\lambda_r = 0$ for $r=1, \ldots, k$). By Lemma \ref{abelian}, the same statements will hold in the original torsion-free abelian group $A$, thus completing the proof. 

Suppose first that $c=0$. Note that since $| \det B| = 1$, this implies that $a = d = 1$, so that equations (\ref{equation4additive}),  (\ref{equation3additive}) are reduced to
\begin{alignat}{5}
& \lambda_1 + b\mu_{1,s} = \ldots = \lambda_k + b \mu_{k,s} \> \text{ for }s=1, \ldots, \ell, \label{equation4mappingtorus} \\
& \sum_{r=1}^k n_{r,s}\mu_{r,s} = 0 \> \text{ for }s=1, \ldots, \ell. \label{equation3mappingtorus} 
\end{alignat} 
We show by induction on $k+\ell$ that, in this case, $\mu_{r,s} = 0$ if $n_{r,s} > 0$. The base case $k+\ell = 2$ is trivial. By the symmetry of equations (\ref{equation2additive}), (\ref{equation3mappingtorus}), (\ref{equation4mappingtorus}), we may assume that $\mu_{k, 1} \geq \mu_{r,s}$ for all~$r$ and $s$, and that $\mu_{k, 1} \geq \ldots \geq \mu_{k, \ell}$. Note that the former implies that in particular $\mu_{k,1} \geq \mu_{r,1}$, so we obtain from (\ref{equation4mappingtorus}) that $\mu_{k, \ell} \geq \mu_{r, \ell}$ for $r = 1, \ldots, k$.   If $\mu_{k, \ell} \geq 0$, then since $\sum_{s=1}^\ell n_{k,s}\mu_{k,s} = 0$, we must have $n_{k,1}\mu_{k,1} = \ldots = n_{k,\ell}\mu_{k,\ell} = 0$. This implies that $\mu_{k,s} = 0$ if $n_{k,s} > 0$, so we may apply the induction hypothesis to the system of equations
\begin{alignat*}{5}
& \lambda_1 + b\mu_{1,s} = \ldots = \lambda_{k-1} + b \mu_{k-1,s} \> \text{ for }s=1, \ldots, \ell, \\
& \sum_{s = 1}^\ell n_{r,s}\mu_{r,s} = 0 \> \text{ for }r=1, \ldots, k-1, \\
& \sum_{r=1}^{k-1} n_{r,s}\mu_{r,s} = 0 \> \text{ for }s=1, \ldots, \ell. 
\end{alignat*}
Now suppose that $\mu_{k, \ell} < 0$. Since $\mu_{k, \ell} \geq \mu_{r, \ell}$ for $r = 1, \ldots, k$ and $\sum_{r=1}^k n_{r, \ell}\mu_{r, \ell} = 0$, we have that $n_{1,\ell}\mu_{1, \ell} = \ldots = n_{k, \ell}\mu_{k, \ell} = 0$. This implies that $\mu_{r, \ell} = 0$ if $n_{r, \ell} > 0$, so we may apply the induction hypothesis to the system of equations
\begin{alignat*}{5}
& \lambda_1 + b\mu_{1,s} = \ldots = \lambda_k + b \mu_{k,s} \> \text{ for }s=1, \ldots, \ell-1, \\
& \sum_{s = 1}^{\ell-1} n_{r,s}\mu_{r,s} = 0 \> \text{ for }r=1, \ldots, k, \\
& \sum_{r=1}^k n_{r,s}\mu_{r,s} = 0 \> \text{ for }s=1, \ldots, \ell-1. 
\end{alignat*}
This completes the proof for the case $c = 0$.

We assume for the remainder of the proof that $c > 0$. Define
\begin{alignat*}{4}
N = \begin{pmatrix} n_{1,1} & \ldots & n_{k,1} \\ \vdots & \ddots & \vdots \\ n_{1,\ell} & \ldots & n_{k,\ell} \end{pmatrix}, \quad {\bf u} = \begin{pmatrix} a \lambda_1 + b\mu_{1,1} \\ \vdots \\ a\lambda_1 + b\mu_{1, \ell} \end{pmatrix}, \quad {\bf w} =  \begin{pmatrix} c \lambda_1 \\ \vdots \\ c \lambda_k \end{pmatrix}
\end{alignat*}
and let $N_r$ be the $r^\text{th}$ column of $N$. We have
\begin{alignat*}{4}
{\bf u}^T N_r &= (a\lambda_r + b \mu_{r,1}, \ldots, a \lambda_r + b \mu_{r, \ell}) N_r \\
&= (a\lambda_r, \ldots, a \lambda_r)N_r + b (\mu_{r,1}, \ldots, \mu_{r, \ell}) N_r \\
&= (a\lambda_r, \ldots, a \lambda_r)N_r \\
&= \sum_{s=1}^\ell   n_{r,s} a \lambda_r
\end{alignat*}
where the first equality follows from (\ref{equation4additive}) and the third follows from (\ref{equation2additive}). Thus,
\begin{equation}\label{uNw1}
{\bf u}^T N {\bf w} = \left (\sum_{s=1}^\ell   n_{1,s} a \lambda_1, \ldots, \sum_{s=1}^\ell  n_{k,s} a \lambda_k \right) {\bf w} = \sum_{r,s} n_{r,s} ac \lambda_r^2.
\end{equation}
On the other hand, we have
\begin{alignat*}{4}
N {\bf w} = \begin{pmatrix} \sum_{r=1}^k n_{r,1}c \lambda_r \\ \vdots \\ \sum_{r=1}^k n_{r, \ell}c \lambda_r \end{pmatrix} = - \begin{pmatrix} \sum_{r=1}^k n_{r,1}d\mu_{r,1} \\ \vdots \\ \sum_{r=1}^k n_{r, \ell} d \mu_{r, \ell}\end{pmatrix}
\end{alignat*}
where the second equality follows from (\ref{equation3additive}). It follows that
\begin{alignat}{5}
- {\bf u}^T N {\bf w} = {\bf u}^T \begin{pmatrix} \sum_{r=1}^k n_{r,1}d\mu_{r,1} \\ \vdots \\ \sum_{r=1}^k n_{r, \ell} d \mu_{r, \ell} \end{pmatrix} = \sum_{r,s} n_{r,s} (a \lambda_1 + b \mu_{1, s})d\mu_{r,s} = \sum_{r,s} n_{r,s} (a \lambda_r + b\mu_{r,s})d\mu_{r,s} \label{uNw2}
\end{alignat}
where the last equality follows from (\ref{equation4additive}). Combining (\ref{uNw1}) and (\ref{uNw2}), we obtain
\begin{equation}\label{final}
0 = \sum_{r,s} n_{r,s} ac \lambda_r^2 + \sum_{r,s}n_{r,s} (a \lambda_r + b\mu_{r,s})d\mu_{r,s} = \sum_{r,s} n_{r,s} (bd\mu_{r,s}^2 + ad \lambda_r \mu_{r,s} + ac\lambda_r^2).
\end{equation}

We claim that $bd\mu_{r,s}^2 + ad \lambda_r \mu_{r,s} + ac\lambda_r^2 \geq 0$ for any $r$ and $s$. If $d=0$, this is clear. Otherwise, we may view $bd\mu_{r,s}^2 + ad \lambda_r \mu_{r,s} + ac\lambda_r^2$ as a quadratic polynomial in $\mu_{r,s}$ with positive leading coefficient $bd$ and discriminant
\[
\Delta_r = (ad-4bc) ad\lambda_r^2 = (\det B - 3bc)ad\lambda_r^2.
\]
Since $|\det B| = 1$, we have that $\Delta_r \leq 0$, and so $bd\mu_{r,s}^2 + ad \lambda_r \mu_{r,s} + ac\lambda_r^2 \geq 0$.

Now let $r \in \{1, \ldots, k\}$. We show that $\lambda_r = 0$. Indeed, we have $n_{r,s} > 0$ for some $s$ since $\sum_{s=1}^k n_{r,s} = n_r > 0$. Thus, by (\ref{final}) and the previous paragraph, we have \[bd\mu_{r,s}^2 + ad \lambda_r \mu_{r,s} + ac\lambda_r^2 = 0.\] 
If $d = 0$, this immediately implies that $\lambda_r = 0$. Now suppose $d, \lambda_r > 0$. Then $\Delta_r < 0$ and so $bd\mu_{r,s}^2 + ad \lambda_r \mu_{r,s} + ac\lambda_r^2 > 0$, a contradiction.
\end{proof}

\begin{remark} Note that if $c,d > 0$, we also obtain that $\mu_{r,s} = 0$ whenever $n_{r,s} > 0$, so that $\rho(z)$ is also a virtually unipotent matrix. Thus, if all the entries of a gluing matrix for a manifold $M \in \mathfrak{E}$ are nonzero, then any element of $\pi_1(M)$ representing a curve on the JSJ torus of $M$ is virtually unipotent. 
\end{remark}

\bibliographystyle{amsalpha}
\bibliography{biblio}

\providecommand{\bysame}{\leavevmode\hbox to3em{\hrulefill}\thinspace}
\providecommand{\MR}{\relax\ifhmode\unskip\space\fi MR }
\providecommand{\MRhref}[2]{%
  \href{http://www.ams.org/mathscinet-getitem?mr=#1}{#2}
}
\providecommand{\href}[2]{#2}
\begin{thebibliography}{AGM13}

\bibitem[AGM13]{agol2013virtual}
Ian Agol, Daniel Groves, and Jason Manning, \emph{The virtual {Haken}
  conjecture}, Doc. Math \textbf{18} (2013), 1045--1087.

\bibitem[Ago18]{315430}
Ian Agol, \emph{Hyperbolic $3$-manifold groups that embed in compact {Lie}
  groups}, MathOverflow, 2018, URL:https://mathoverflow.net/q/315430 (version:
  2018-11-16).

\bibitem[BG19]{berlai2019linearity}
Federico Berlai and Javier de la~Nuez Gonz{\'a}lez, \emph{Linearity of graph
  products}, arXiv preprint arXiv:1906.11958 (2019).

\bibitem[BK95]{buyalo1995geometrization}
Sergey~V Buyalo and Viktor~L Kobel'skii, \emph{Geometrization of
  graph-manifolds. {II. Isometric} geometrization}, Algebra i Analiz \textbf{7}
  (1995), no.~3, 96--117.

\bibitem[But17a]{button2017free}
Jack~O Button, \emph{Free by cyclic groups and linear groups with restricted
  unipotent elements}, Groups Complexity Cryptology \textbf{9} (2017), no.~2,
  137--149.

\bibitem[But17b]{button2017properties}
\bysame, \emph{Properties of linear groups with restricted unipotent elements},
  arXiv preprint arXiv:1703.05553 (2017).

\bibitem[But19]{button2019aspects}
\bysame, \emph{Aspects of non positive curvature for linear groups with no
  infinite order unipotents}, Groups, Geometry, and Dynamics \textbf{13}
  (2019), no.~1, 277--293.

\bibitem[Ebe82]{eberlein1982canonical}
Patrick Eberlein, \emph{A canonical form for compact nonpositively curved
  manifolds whose fundamental groups have nontrivial center}, Mathematische
  Annalen \textbf{260} (1982), no.~1, 23--29.

\bibitem[FLS11]{frigerio2011rigidity}
Roberto Frigerio, Jean-Fran{\c{c}}ois Lafont, and Alessandro Sisto,
  \emph{Rigidity of high dimensional graph manifolds}, arXiv preprint
  arXiv:1107.2019 (2011).

\bibitem[GW71]{gromoll1971some}
Detlef Gromoll and Joseph~A Wolf, \emph{Some relations between the metric
  structure and the algebraic structure of the fundamental group in manifolds
  of nonpositive curvature}, Bulletin of the American Mathematical Society
  \textbf{77} (1971), no.~4, 545--552.

\bibitem[Hem87]{hempel1987residual}
John Hempel, \emph{Residual finiteness for 3-manifolds}, Combinatorial group
  theory and topology (Alta, Utah, 1984) \textbf{111} (1987), 379--396.

\bibitem[Kap01]{kapovich2009hyperbolic}
Michael Kapovich, \emph{{Hyperbolic Manifolds and Discrete Groups}},
  Birkh\"auser, Boston, 2001.

\bibitem[KL96]{kapovich1996actions}
Michael Kapovich and Bernhard Leeb, \emph{Actions of discrete groups on
  nonpositively curved spaces}, Mathematische Annalen \textbf{306} (1996),
  no.~1, 341--352.

\bibitem[KL98]{kapovich19983}
\bysame, \emph{3-manifold groups and nonpositive curvature}, Geometric \&
  Functional Analysis GAFA \textbf{8} (1998), no.~5, 841--852.

\bibitem[Lee92]{leeb1992manifolds}
Bernhard Leeb, \emph{3-manifolds with(out) metrics of nonpositive curvature},
  Ph.D. thesis, University of Maryland, 1992.

\bibitem[Liu13]{liu2013virtual}
Yi~Liu, \emph{Virtual cubulation of nonpositively curved graph manifolds},
  Journal of Topology \textbf{6} (2013), no.~4, 793--822.

\bibitem[LMR00]{lubotzky2000word}
Alexander Lubotzky, Shahar Mozes, and Madabusi~S Raghunathan, \emph{The word
  and {Riemannian} metrics on lattices of semisimple groups}, Publications
  Math{\'e}matiques de l'Institut des Hautes {\'E}tudes Scientifiques
  \textbf{91} (2000), no.~1, 5--53.

\bibitem[Min19]{10.1093/imrn/rnz249}
Ashot Minasyan, \emph{Virtual retraction properties in groups}, International
  Mathematics Research Notices (2019), rnz249.

\bibitem[PW18]{przytycki2018mixed}
Piotr Przytycki and Daniel Wise, \emph{Mixed 3-manifolds are virtually
  special}, Journal of the American Mathematical Society \textbf{31} (2018),
  no.~2, 319--347.

\bibitem[RR00]{radjavi2012simultaneous}
Heydar Radjavi and Peter Rosenthal, \emph{Simultaneous {Triangularization}},
  Springer, New York, 2000.

\bibitem[Yau71]{yau1971fundamental}
Shing~Tung Yau, \emph{On the fundamental group of compact manifolds of
  non-positive curvature}, Annals of Mathematics (1971), 579--585.

\end{thebibliography}

\end{document}